\documentclass[12pt]{article}
\input epsf.tex

\usepackage{graphicx}
\usepackage{amsmath,amsthm,amsfonts,amscd,amssymb,comment,eucal,latexsym,mathrsfs}
\usepackage{stmaryrd}
\usepackage[all]{xy}

\usepackage{epsfig}

\usepackage[all]{xy}
\xyoption{poly}
\usepackage{fancyhdr}
\usepackage{wrapfig}
\usepackage{epsfig}
\usepackage{mathtools}
\usepackage[pdftex]{hyperref}

\theoremstyle{plain}
\newtheorem{thm}{Theorem}[section]

\newtheorem{lem}[thm]{Lemma}

\newtheorem{claim}{Claim}
\theoremstyle{definition}
\newtheorem{defn}{Definition}
\theoremstyle{remark}
\newtheorem{remark}{Remark}

\advance \topmargin by -\headheight
\advance \topmargin by -\headsep
\evensidemargin \oddsidemargin
\newcommand*{\A}{\mathtt{A}}
\newcommand*{\B}{\mathtt{B}}
\newcommand*{\C}{\mathtt{C}}
\newcommand*{\ball}[2]{\mathrm{B}(#1,#2)} 
\newcommand*{\shent}{\mathrm{H}} 
\newcommand*{\f}{f} 

\let\mb\mathnormal
\newcommand*{\st}{\,:\,}
\DeclarePairedDelimiter{\abs}{\lvert}{\rvert}

    \def\EE{{\mathbb{E}}}         \def\NN{{\mathbb{N}}}   \def\Q{{\mathbb{Q}}} \def\R{{\mathbb{R}}}        \def\ZZ{{\mathbb{Z}}}

 \def\cB{{\mathcal{B}}}   \def\cE{{\mathcal{E}}} \def\cF{{\mathcal{F}}}         \def\cO{{\mathcal{O}}}      \def\cU{{\mathcal{U}}} \def\cV{{\mathcal{V}}} \def\cW{{\mathcal{W}}}   

    \def\tcE{{\tilde{\mathcal{E}}}} \def\tcF{{\tilde{\mathcal{F}}}}

\def\hal{{\hat{\alpha}}}
\def\hbe{{\hat{\beta}}}

                         \def\tZ{{\tilde{Z}}}

    \def\tcE{{\tilde{\cE}}} \def\tcF{{\tilde{\cF}}}

       \def\tGamma{{\widetilde{\Gamma}}}  \def\tTheta{{\widetilde{\Theta}}}    

\newcommand{\G}{\Gamma}

\newcommand{\Om}{\Omega}

\newcommand{\eps}{\epsilon}

\renewcommand\a{\alpha}
\renewcommand\b{\beta}
\renewcommand\d{\delta}
\newcommand\g{\gamma}

\newcommand\s{\sigma}

\newcommand\Hom{\operatorname{Hom}}

\newcommand\Past{\operatorname{Past}}
\newcommand\Prob{\operatorname{Prob}}

\newcommand\sym{\operatorname{sym}}

\newcommand\Unif{\operatorname{Unif}}

\def\cc{{\curvearrowright}}

  \newcommand{\rL}{\textrm{L}}
  
\begin{document}
\title{Entropy for actions of free groups under bounded orbit-equivalence}
\author{Lewis Bowen\footnote{supported in part by NSF grant DMS-1900386} and Yuqing (Frank) Lin\\ University of Texas at Austin}
\maketitle

\begin{abstract}
The $f$-invariant is a notion of entropy for probability-measure-preserving actions of free groups. We show it is invariant under bounded orbit-equivalence.
\end{abstract}


\noindent
\tableofcontents

\section{Introduction}
The goal of this note is to prove that the $f$-invariant, which is a generalization of Kolmogorov-Sinai entropy to actions of free groups, is invariant under $\rL^1$-orbit-equivalence. We begin by quickly going over the history of entropy theory and orbit-equivalence to set notation and motivate the main result.

\subsection{Notation}
Throughout, we will denote standard probability spaces by $(X,\mu)$, $(Y,\nu)$, etc and their Borel sigma-algebras by $\cB_X, \cB_Y$ etc. If $G$ is a countable group then a {\bf pmp action} of $G$ is a tuple $(X,\mu, T,G)$ where $(X,\mu)$ is a standard probability space and for every $g\in G$, $T^g:X \to X$ is a measure-preserving transformation satisfying $T^{gh}=T^g\circ T^h$ (pmp is short for probability-measure-preserving). In the special case that $G=\ZZ$, a pmp action is also called a {\bf pmp transformation}.

Let $(Y,\nu,U,H)$ be another pmp action where $H$ is a countable group, $(Y,\nu)$ a standard probability space and $U = (U^h)_{h \in H}$ a pmp action of $H$ on $Y$. Let $\Phi:(X,\mu) \to (Y,\nu)$ be a measure-space isomorphism.
\begin{itemize}
\item $\Phi$ is an {\bf orbit-equivalence} (OE) if for a.e. $x \in X$, 
$$\Phi(T^G x)=U^H \Phi(x)$$
where, for example, $T^G x= \{T^g x:~g \in G\}$ is the $T^G$-orbit of $x$.
\item $\Phi$ is a {\bf measure-conjugacy} if $G=H$ and $\Phi(T^gx)=U^g\Phi(x)$ for a.e. $x\in X$ and every $g\in G$. 
\end{itemize}
Measure-conjugacy implies orbit-equivalence and the converse is false in general.

\subsection{Classical entropy theory}


Classical entropy theory was introduced by Kolmogorov in order to classify Bernoulli shifts up to measure conjugacy \cite{kolmogorov-1958, kolmogorov-1959}. Given a countable group $G$ and a Borel probability space $(K,\kappa)$, the {\bf Bernoulli shift over $G$ with base $(K,\kappa)$} is the product probability space $(K^G,\kappa^G)$ with the $G$-action $G \cc K^G$ given by
$$(g \cdot x)_f = x(g^{-1}f),\quad g,f \in G, x\in K^G.$$
For example, if $x$ is a random element of $K^G$ with law $\kappa^G$ then $(x_g)_{g\in G}$ is a $G$-indexed i.i.d. process in which each variable $x_g$ has law $\kappa$.

The {\bf Shannon entropy} of the base space is defined by
$$H(K,\kappa) = -\sum_{k \in K} \kappa(k)\log \kappa(k)$$
if $\kappa$ is concentrated on a countable set. Otherwise, $H(K,\kappa):=+\infty$. 

Now suppose $(X,\mu)$ is a standard probability space. An {\bf observable} on $X$ is a measurable map $\phi:X \to \A$ where  $\A$ is a finite or countable set. If $\psi: X \to \B$ is another observable, then the {\bf join} of $\phi$ and $\psi$ is the observable
$$\phi\vee \psi:X \to \A\times \B, \quad \phi\vee\psi(x)=(\phi(x),\psi(x)).$$
The {\bf Shannon entropy} of the observable $\phi$ with respect to the measure $\mu$ is 
$$H_{\mu}(\phi) = H(\A, \phi_*\mu) = -\sum_{a \in \A} \mu(\phi^{-1}(a))\log \mu(\phi^{-1}(a)).$$
If $T:X \to X$ is measure-preserving then the {\bf entropy rate of $(X,\mu,T,\phi)$} is
$$h_{\mu}(T,\phi) := \lim_{n\to\infty} \frac{1}{n} H_{\mu}( \phi \vee \phi \circ T \vee \cdots \vee \phi \circ T^{n-1}).$$
The {\bf entropy rate} of $T$ is defined to be the supremum of entropy rates $h_{\mu}(T,\phi)$ over all finite Shannon entropy observables $\phi$:
$$h_{\mu}(T) =\sup\{h_{\mu}(T,\phi):~ H_{\mu}(\phi)<\infty\}.$$
This is a measure-conjugacy invariant. Moreover, Kolmogorov proved the entropy rate of a Bernoulli shift over $\ZZ$ is the same as the Shannon entropy of its base space \cite{kolmogorov-1958, kolmogorov-1959}. This proves one direction of Ornstein's Isomorphism Theorem \cite{ornstein-1970a}: Bernoulli shifts over $\ZZ$ are measurably conjugate if and only if they have the same base space entropy. All of these results were extended to actions of countably infinite amenable groups \cite{OW80, OW87}. For example, if $G$ is a countable amenable group then two Bernoulli shifts over $G$ are measurably conjugate if and only if they have the same base space entropy.

\subsection{Orbit equivalence (the amenable case)}

In \cite{MR0131516, MR0158048}, Dye proved that all ergodic aperiodic pmp transformations are orbit-equivalent (where aperiodic means that a.e. orbit is infinite).  In \cite{OW80} it was announced that all essentially free ergodic  pmp actions of countably infinite amenable groups are OE. A complete proof appears in \cite{MR662736}. In particular, entropy is not an OE-invariant.

\subsection{Quantitative orbit equivalence}
Suppose that $\Phi:X \to Y$ is an orbit-equivalence as above. Also suppose both actions are essentially free. Then there are cocycles $\a: G\times X \to H$, $\beta:H \times Y \to G$ defined by
$$U^{\a(g,x)}\Phi(x) = \Phi(T^gx), \quad T^{\beta(h,y)}\Phi^{-1}(y) = \Phi^{-1}(U^h y).$$
These cocycles satisfy the identities
\begin{eqnarray}\label{E:cocycle1}
\a(gh,x) = \a(g,T^hx)\a(h,x), &\quad& \beta(gh,y) = \beta(g,U^hy)\beta(h,y) \label{E:cocycle1} \\
\a(\beta(g,y),\Phi^{-1}(y))=g, &\quad& \beta(\a(g,x),\Phi(x))=g. \label{E:cocycle2}
\end{eqnarray}
Suppose $G$ and $H$ are both finitely generated groups. After choosing finite generating sets, we may let $|\cdot|_G:G \to \R$, $|\cdot|_H:H\to \R$ denote word-length functions. The map $\Phi$ is said to be an {\bf $\rL^p$- orbit equivalence} ($\rL^p$-OE) if for every $g \in G$ and $h \in H$ the functions
$$x \mapsto |\a(g,x)|_H, \quad y \mapsto  |\beta(h,y)|_G$$
are in $\rL^p(X,\mu)$ and $\rL^p(Y,\nu)$ respectively. $\rL^1$-OE is also called {\bf integrable OE} and $\rL^\infty$-OE is also called {\bf bounded OE}. These notions do not depend on the choice of finite generating sets. In \cite{MR3579704}, Austin proved that entropy is an $\rL^1$-orbit equivalence invariant for actions of finitely generated amenable groups.


The map $\Phi$ is said to a {\bf Shannon orbit equivalence} if for every $g \in G$ and $h \in H$, the observables $\a(g,\cdot):X \to H$ and $\beta(h,\cdot):Y \to G$ have finite Shannon entropies.  In \cite{MR4297192} Kerr and Li proved that if each of $G$ and $H$ contain a w-normal amenable subgroup that is neither locally finite nor virtually cyclic then entropy for their actions is invariant under Shannon orbit equivalence, where entropy means maximum sofic entropy. The groups are not required to be amenable or finitely generated but they are required to be sofic. The statement is false for locally finite groups by a counterexample due to Vershik \cite{MR1304093}.  In fact, Vershik's counterexample is with a bounded orbit-equivalence (but using non-finitely generated groups). 

\subsection{Entropy for non-amenable groups}
For a long time, there was no entropy theory for actions of non-amenable groups. Ornstein and Weiss exhibited an example of a factor map between Bernoulli shifts over a non-abelian free group in which the base space entropy of the source is smaller than the base space entropy of the target \cite{OW87}. By contrast, entropy for amenable groups cannot increase under a factor map. In spite of this it is possible to define entropy for actions of non-amenable groups. Today there are several versions of entropy: sofic, Rokhlin, naive, etc (see \cite{MR4138907} for a survey). We will focus on the $f$-invariant which is a flavor of entropy specifically tailored to actions of free groups.

\subsection{The $f$-invariant}\label{S:intro-f}
Let $G = \langle S \rangle$ denote the rank $r$ free group with generating set $S = \{s_1, \ldots, s_r\}$. Let $(X,\mu,T,G)$ be a pmp action of $G$ and $\phi:X \to \A$ be an observable. For any subset $H \subset G$, let $\phi^{T,H}:X \to \A^H$ be the join
$$\phi^{T,H} = \bigvee_{h \in H} \phi \circ T^{h^{-1}}.$$
Let $\ball{e}{\rho}\subset G$ denote the ball of radius $\rho$ centered at the identity in $G$ with respect to the word metric (where $e \in G$ denotes the identity). Write $\phi^{T,\rho} = \phi^{T,\ball{e}{\rho}}$. 

Define
\begin{align*}
	F_\mu (T, \phi) &= (1-2r) \shent_\mu (\phi) + \sum_{i=1}^r \shent_\mu (\phi^{T,\{e,s_i\}}) \\
	\f_\mu (T, \phi) &= \inf_\rho F_\mu (T, \phi^{T,\rho}) = \lim_{\rho \to \infty} F_\mu (T, \phi^{T,\rho}).
\end{align*}

An observable $\phi$ is {\bf $T$-generating} if the smallest $T^G$-invariant Borel sigma-algebra in which $\phi$ is measurable is the full Borel sigma-algebra $\cB_X$, up to sets of measure zero. If there exists an observable $\phi$ which has finite Shannon entropy $H_{\mu}(\phi)<\infty$ and is $T$-generating then the action $(X,\mu,T,G)$ is said to be {\bf finitely generated}. This terminology is justified by Seward's generalization of Krieger's Theorem \cite{seward-kreiger-1} which implies that if $(X,\mu,T,G)$ is finitely generated then there exists a $T$-generating observable $\phi:X \to \A$ such that $\A$ is finite. 

The main theorem of \cite{MR2630067} is that if $(X,\mu,T,G)$ is finitely generated then there is a number $\f_{\mu}(T) \in [-\infty, \infty)$ called the {\bf $f$-invariant} such that every $T$-generating finite Shannon entropy observable $\phi$ satisfies $\f_{\mu}(T,\phi)=\f_{\mu}(T)$.  If the action is not finitely generated then the $f$-invariant is not defined.

In \cite{MR2630067}, it is shown that the $f$-invariant of a Bernoulli shift action equals the Shannon entropy of its base space. So the $f$-invariant distinguishes Bernoulli shifts. However, in \cite{MR2763777} it is shown that all Bernoulli shifts of a free group are OE. In particular, the $f$-invariant is not an OE-invariant.

The main theorem of this paper is:
\begin{thm}\label{thm:main}
The $f$-invariant is invariant under bounded-orbit-equivalence. To be precise, suppose $G$ is a free group and  $(X,\mu,T,G)$, $(Y,\nu,U,G)$ are finitely generated essentially free pmp actions of $G$. If these actions are bounded orbit-equivalent then $f_{\mu}(T) = f_{\nu}(U)$. 
\end{thm}

\begin{remark}
Belinskaya proved that if two ergodic aperiodic pmp transformations $T, U$ are $\rL^1$-OE then they are either measurably conjugate or flip-conjugate (which means $T$ is measurably conjugate to $U^{-1}$) \cite{MR0245756}. So if $G=\ZZ$ then Theorem \ref{thm:main} is trivial. This motivates the question: if $G,T,U$ are as in Theorem \ref{thm:main}, then must $T$ be measurably conjugate to $U \circ \a$ for some automorphism $\a:G \to G$? A recent work-in-progress due to Matthieu Joseph shows the answer is `no' by an explicit counterexample. 
\end{remark}

\subsection{Related literature}

About the problem of classifying Bernoulli shifts up to measure-conjugacy: Seward proved that if two probability spaces have the same Shannon entropy then the corresponding Bernoulli shifts are measurably conjugate, for any countably infinite group \cite{seward-ornstein}. The converse holds for sofic groups \cite{bowen-jams-2010, MR2813530}.

About the problem of classifying Bernoulli shifts up to OE: if a group $G$ is Bernoulli cocycle-superrigid, then it is immediate that if two Bernoulli shifts over $G$ are OE then they are measurably-conjugate. This notion is implicit in ground-breaking work of Popa where it is proven that $G$ is Bernoulli cocycle-superrigid if it contains an infinite normal subgroup $N$ such that either (i) the pair $(G,N)$ has relative property (T), or (ii) $N$ is generated by (element-wise) commuting subgroups $H$ and $K$, with $H$ nonamenable and $K$ infinite (\cite[Theorem 0.1]{Pop07} and \cite[Theorem 4.1]{Pop08}). Recent work show that a type of entropy called weak Pinsker entropy  is OE-invariant for all essentially free pmp actions of Bernoulli cocycle-superrigid groups  \cite{bowen-td2}. On the other hand, there are two classes of groups for which it is known that all Bernoulli shifts are OE. These are countably infinite amenable groups \cite{OW80, MR662736} and free products of amenable groups \cite{MR2763777}. It remain a very interesting open problem whether fundamental groups of closed surfaces of genus $\ge 2$ have this property.

Kammeyer and Rudolph found a unified approach to Dye's and Ornstein's Theorem which also gives explicit restrictions on orbit-equivalence which imply entropy-invariance \cite{kammeyer-rudolph}, for actions of amenable groups.

Kerr and Li find conditions under which topological sofic entropy is preserved under continuous orbit equivalence in \cite{kerr-li-oe1}. 

Rudolph and Weiss proved that if $T$ and $S$ are orbit-equivalent pmp ergodic essentially free actions of amenable groups, then their entropies {\em relative to their orbit-change sigma-algebras} are equal \cite{rudolph-weiss-2000}. This insight was developed into a technique for generalizing entropy-theory results for $\ZZ$-actions to actions of arbitrary amenable groups \cite{danilenko-2001, danilenko-2002}.

\subsection{Outline of the proof and paper}

The techniques we use are completely different from the works of Austin \cite{MR3579704} and Kerr-Li \cite{MR4297192}, \cite{kerr-li-oe1}. 

By \cite{bowen-entropy-2010b}, we know that the $f$-invariant is the exponential growth rate of the average number of approximate periodic points (or microstates) which are approximately equidistributed with respect to the given measure. This result is strengthened in \S \ref{S:symbolic} for the special case in which the invariant measure $\mu$ is supported on a subshift of finite type (SFT). In this case, we find that, to compute the $f$-invariant, it suffices to count actual periodic points in the SFT.  This is very special to the free group. For example, the analogous statement is false for actions of $\ZZ^2$, because of the existence of SFTs without periodic points. 

In \S \ref{S:prelim} we set notation for the rest of the paper. In particular, we assume $(X,\mu)=(Y,\nu)$ and $T, U$ are essentially free actions of $G$ on $(X,\mu)$ with the same orbits. There are cocycles $\a: G\times X \to G$, $\beta:G \times X \to G$ defined by
$$U^{\a(g,x)}x = T^gx, \quad T^{\beta(g,x)}x = U^g x.$$

The idea now is to show that, after replacing the system $(X,\mu,T,G)$ with a  measurably conjugate system, we may assume $X$ is a SFT. Moreover, we can design the SFT so that $\alpha(g,\cdot)$ and $\beta(g,\cdot)$ are continuous functions of $X$. Then any periodic point for the system $(X,\mu,T,G)$ can be re-arranged to obtain a periodic point for $(X,\mu,U,G)$ (and vice versa). Because the $f$-invariant is the exponential growth rate of approximately-equidistributed periodic points, this proves the main Theorem \ref{thm:main}. This is shown in \S \ref{S:orbit-change} and  \S \ref{S:SFTBOE}. The last section \S \ref{S:open} is devoted to open problems.


{\bf Acknowledgements}. L.B. would like to thank David Kerr for helpful conversations.

\section{Symbolic dynamics}\label{S:symbolic}

Let $\A$ be a countable or finite alphabet, $\A^G$ be the set of functions $x:G \to \A$ and $G$ act on $\A^G$ by 
$$(gx)(f)=x(g^{-1}f)\quad \forall f,g \in G, x\in \A^G.$$
Symbolic dynamics is the study of $G$-invariant measures and subspaces of $\A^G$. 

Whenever we are working with symbolic dynamical systems we always use the left shift action described above and  as seen above we do not use a symbol such as $T$ for the action. We also write $f_\mu(\phi)$ instead of $f_\mu(T,\phi)$ for example and we denote the system by $(\A^G,\mu,G)$.

In the first subsection below, we recall a formula for the $f$-invariant of an invariant measure $\mu$ on $\A^G$ in terms of counting periodic points. In the second subsection, we assume the support of $\mu$ is contained in a subshift of finite type $Z \subset \A^G$ and prove a formula for the $f$-invariant in terms of counting periodic points {\em which lie in $Z$}. This second formula is crucial to our proof of Theorem \ref{thm:main}. It seems to be a very special fact about free groups. For example, the analogous statement fails for the group $\ZZ^d$ for any $d\ge 2$ because of the existence of subshifts of finite type which contain no periodic points.


\subsection{The $f$-invariant via periodic points}


In this section let $\A$ be a finite set.  First we approximate the action of $G$ on itself by an action of $G$ on a finite set. So let $\s:G \to \sym(n)$ be a homomorphism into the symmetric group on $[n]=\{1,\ldots, n\}$.

Next we consider observables $\mb{x}:[n] \to \A$ whose local statistics approximate $\mu$. To make this precise, define the {\bf pullback name of $\mb{x}$ at vertex $v\in [n]$} by
$$\mb{x}^\s_v \in \A^G, \quad \mb{x}^\s_v(g) = \mb{x}(\s(g)^{-1}v).$$
We observe that the map $v \mapsto \mb{x}^\s_v$ is equivariant in the sense that
$$\mb{x}^\s_{\s(g)v} = g \mb{x}^\s_v$$
for any $g\in G$. In particular, $\mb{x}^\s_v$ is a periodic point of $\A^G$ (that is, it has a finite $G$-orbit). 

The {\bf empirical distribution} of $\mb{x}$ is defined by
$$P^\sigma_{\mb{x}} = \frac{1}{n} \sum_{v\in [n]} \d_{\mb{x}^\s_v} \in \Prob(\A^G)$$
where $ \d_{\mb{x}^\s_v}$ is the Dirac probability measure concentrated on  ${\mb{x}^\s_v}$ and  $\Prob(\A^G)$ is the space of all Borel probability measures on $\A^G$. 

 Informally, we consider $\mb{x}$ to be a good approximation to $\mu$ if $P^\sigma_{\mb{x}}$ is close to $\mu$. To make this notion precise, recall that the weak* topology on $\Prob(\A^G)$ is the weakest topology with the following property: for every continuous function $f: \A^G \to \R$, the map $\mu \mapsto \int f~d\mu$ is a continuous function on $\Prob(\A^G)$ with respect to the weak* topology. Thus a sequence $(\mu_i)_i$ weak* converges to a measure $\mu_\infty$ if and only if: for every continuous $f: \A^G \to \R$, $\int f~d\mu_i \to \int f~d\mu_\infty$ as $i\to\infty$. By the Banach-Alaoglu Theorem, $\Prob(\A^G)$ is compact. 
 
 Now let $\cO \subset \Prob(\A^G)$ be a weak* open neighborhood of $\mu$ and define
\[ \Omega(\sigma, \cO) := \{ \mb{x} \in \A^n \st P^\sigma_{\mb{x}} \in \cO \}. \]
Then $\Omega(\sigma,\cO)$ is the set of all observables on $[n]$ whose empirical distributions are in $\cO$. 

For each $n \in \NN$, let  $u_n = \Unif(\Hom(G, \sym(n)))$ be the uniform probability measure on the set of homomorphisms from $G$ to $\sym(n)$. The main result of \cite{bowen-entropy-2010b} is the formula
\begin{eqnarray}\label{E:fsofic}
 f_\mu (\A^G) = \inf_{\cO \ni \mu} \limsup_{n \to \infty} \frac{1}{n} \log \EE_{\sigma \sim u_n} \abs{\Omega(\sigma, \cO)}.
 \end{eqnarray}
The goal of the next section is to make a small but crucial change to the formula above in the special case in which the support of $\mu$ is contained in a subshift of finite type $Z \subset \A^G$. 


\subsection{Subshifts of finite type}

\begin{defn}[Subshifts of finite type]\label{D:SFT}
As above, let $G$ be a countable group, $\A$ a finite set and $G \cc \A^G$ the left shift action: $(gx)(f)=x(g^{-1}f)$ for $g,f \in G, x\in \A^G$. A subshift $Z \subset \A^G$ is a closed $G$-invariant subspace. It has {\bf finite type} if there exists a finite collection $\cW$ of maps $w: D_w \to \A$ such that
\begin{enumerate}
\item $D_w \subset G$ is finite for all $w\in \cW$;
\item $Z$ is the set of all $x \in \A^G$ satisfying: for all $g \in G$ and $w\in \cW$, the restriction of $gx$ to $D_w$ is not equal to $w$.
\end{enumerate}
So  $\cW$ is a set of forbidden patterns and $Z$ is the subshift of finite type (SFT) {\bf determined by $\cW$}. 

If $G = \langle s_1,...,s_r\rangle$ and for each $w \in \cW$ there exists an $i \in \{1,\ldots,r\}$ such that the domain $D_w=\{e,s_i\}$, then 
we say $Z$ is a {\bf nearest neighbor} subshift of finite type.  In other words a nearest neighbor subshift forbids only certain edge patterns.
\end{defn}

\begin{defn}
Let $Z \subset \A^G$ be a subshift of finite type. With notation as in the previous section, let $\Om_Z(\s,\cO)$ be the set of all $\mb{x} \in \Om(\s,\cO)$ such that $P^\sigma_{\mb{x}}(Z)=1$. This occurs precisely when $\mb{x}^\s_v \in Z$ for all $v \in [n]$.   
\end{defn}

\begin{thm}\label{thm:SFT}
Suppose $\A$ is finite and $Z \subset \A^G$ is a subshift of finite type. Let $\mu\in \Prob(\A^G)$ be a $G$-invariant Borel probability measure concentrated on $Z$ (meaning $\mu(Z)=1$). Let
\[ f^Z_\mu(\A^G)  = \inf_{\cO \ni \mu} \limsup_{n \to \infty} \frac{1}{n} \log \EE_{\sigma \sim u_n} \abs{\Omega_Z(\sigma, \cO)} . \]
Then
	\[ f_\mu (\A^G) = f^Z_\mu(\A^G).\]
\end{thm}
Because $\Omega_Z(\sigma, \cO) \subset \Omega(\sigma, \cO)$, it is immediate that $f^Z_\mu(\A^G) \le f_\mu(\A^G)$. So it suffices to prove the opposite inequality.
\begin{remark}
 Suppose $G$ is a countable group and $Z \subset \A^G$ is a subshift of finite type which has no periodic points (i.e., no points with finite $G$-orbit) but does admit an invariant probability measure. In this case,  no analog of Theorem \ref{thm:SFT} can hold. The first proof that $\ZZ^2$ admits such an SFT is due to Berger \cite{MR216954}. For context, a subshift of finite type that has no periodic points is called {\bf weakly aperiodic}. An SFT on which the group acts freely is called {\bf strongly aperiodic}. There is an interesting line of research whose goal is to determine which groups admit weakly or strongly aperiodic SFTs (e.g. \cite{MR3594268, MR3692905, MR3600067, MR4186473}). 
 \end{remark}

The proof of Theorem \ref{thm:SFT} will take up the rest of this section. We assume from now on that $G = \langle s_1,...,s_r\rangle$ is a free group of rank $r$ and $S= \{s_1,\ldots,s_r\}$.

\subsubsection{Proof sketch} 

In the special case in which $\mu$ is a Markov chain which assigns rational numbers to cylinder sets, the proof follows quickly from \cite{bowen-entropy-2010a, bowen-entropy-2010b}. We obtain the full theorem from by approximating an arbitrary invariant measure by Markov chains. These approximating Markov chains are typically not nearest-neighbor and so we will have to work with general observables $\phi:\A^G \to \C$ with respect to which a measure might be Markov.

We begin by introducing restricted versions of the $f$-invariant and the function $F$. Then we discuss Markov chains and establish a number of lemmas before finishing the proof.
\subsubsection{Restricted versions of $f$ and $F$}\label{S:restricted} 

\emph{We will make the following assumptions for the rest of this section}. As above, we let $\mu$ be a  $G$-invariant Borel probability measure on $\A^G$. Also let $\C$ be a finite set and let $\phi:\A^G \to \C$ be continuous. We do not require that $\phi$ is generating. This induces a map $\Phi: \A^G \to \C^G$ by
$$\Phi(x)(g)= \phi(g^{-1}x).$$
$\Phi$ is the unique $G$-equivariant map from $\A^G$ to $\C^G$ such that $\Phi(x)(e)=\phi(x)$. Let $Y \subset \C^G$ be a nearest neighbor SFT containing the support of $\Phi_*\mu$. 



By (\ref{E:fsofic}),
 \[ f_\mu ( \phi) = \inf_{\cO \ni \Phi_*\mu} \limsup_{n \to \infty} \frac{1}{n} \log \EE_{\sigma \sim u_n} \abs{\Omega(\sigma, \cO)} . \]
We define $f^Y_\mu (\phi)$ by replacing $\Omega(\sigma,\cO)$ with $\Omega_{Y}(\sigma,\cO)$:
\[ f^Y_\mu ( \phi) = \inf_{\cO \ni \Phi_*\mu} \limsup_{n \to \infty} \frac{1}{n} \log \EE_{\sigma \sim u_n} \abs{\Omega_{Y}(\sigma, \cO)} . \]
For example, if  $\phi$ is generating, then to prove Theorem \ref{thm:SFT}, it suffices to show $f_\mu ( \phi)=f^Y_\mu ( \phi)$ with $Y = \Phi(Z)$, along with the  claim below. 

The claim below  verifies that the restricted formula for the $f$-invariant is still invariant for example under recoding maps.
\begin{claim}\label{C : restricted_invar}
Suppose $\A$, $G$, $\phi$, $\C$, $\Phi$ are as above.  Suppose $Z \subset \A^G$ is a subshift, $\mu(Z)=1$, $\Phi(Z) = Y$, and $\Phi$ is a homeomorphism when restricted to $Z$.  Let $\psi: \A^G \to \A$ be the canonical observable (i.e. $\psi(x) = x_e$).  Then $f_\mu^Z(\psi) = f_\mu^Y(\phi)$.
\end{claim}

\begin{proof}
Fix $\s : G \to \sym(n)$ a homomorphism, and $x:[n]\to \A$ an observable.  For every $v \in [n]$ recall the pullback name $x^\s_v\in \A^G$.  Define $y:[n] \to \C$ by $y(v) = \Phi(x^\s_v)_e$.  We claim that $y^\s_v = \Phi(x^\s_v)$.  

First notice that since $\s$ is a homomorphism, for any $g\in G$, $gx_v^\s = x_{\s(g)v}^\s$.  This is because for $h\in G$, $(gx_v^\s)_h = (x_v^\s)_{g^{-1}h} = x(\s(h^{-1}g)v)$ while also $(x_{\s(g)v}^\s)_h = x(\s(h^{-1})(\s(g)v)) = x(\s(h^{-1}g)v)$.  Now $(y^\s_v)_h = y(\s(h^{-1})v) = \Phi(x^\s_{\s(h^{-1})v})_e = h\Phi(x^\s_{\s(h^{-1})v})_h = \Phi(hx^\s_{\s(h^{-1})v})_h = \Phi(x^\s_v)_h$.  Let $\eta: \A^n \to \C^n$ be the map described above (such that $\eta(x) = y$).  $\eta$ is injective on $\Om_Z(\s,\cO)$ for each $n$ and $\s$ because $\Phi$ is injective on $Z$.

Now since $\phi$ is assumed continuous, so is $\Phi$ (by the Curtis-Lyndon-Hedlund theorem).  This implies $\Phi_*: \Prob(\A^G) \to \Prob(\C^G)$ is also continuous.  Let $\nu = \Phi_*\mu$.  Then for any weak* neighborhood $\cU$ of $\nu$ there exists a weak* neighborhood $\cO$ of $\nu$ such that $\Phi_*(\cO) \subset \cU$.  

We will show that for any $\s: G \to \sym(n)$ homomorphism, $\#\Om_Z(\s,\cO) \le \#\Om_Y(\s,\cU)$.  Let $x \in \Om_Z(\s,\cO)$.  This means that $P^\s_x \in \cO$ and $P^\s_x(Z)=1$.  We claim that $P^\s_{\eta(x)} \in \cU$ and $P^\s_{\eta(x)}(Y) = 1$. By definition, $P^\s_{\eta(x)} = (1/n)\sum_{v\in [n]}\d_{\Phi(x^\s_v)} = \Phi_*P^\s_x$, so by continuity of $\Phi_*, P^\s_{\eta(x)} \in \cU$.  Furthermore each $\Phi(x^\s_v) \in Y$ by assumption, so $P^\s_{\eta(x)}(Y) = 1$.  

Now the above claim together with the injectivity of $\eta$ shows that $\#\Om_Z(\s,\cO) \le \#\Om_Y(\s,\cU)$, which in turn shows that $f^Z_\mu(\psi) \le f^Y_\mu(\phi)$.  The proof can be repeated for $\Phi^{-1}$ to show the opposite inequality.
\end{proof}

We need to connect the above definitions with the definition of the $f$-invariant from \S \ref{S:intro-f}. For this, let $\s:G \to \sym(n)$ be a homomorphism and suppose $\psi \in \C^n$ is an observable on $[n]$.  For $H \subset G$ finite, let $\psi^{\s,H} = \vee_{h\in H} \psi \circ \s(h^{-1})$ where $\psi_1 \vee \psi_2(v) = (\psi_1(v), \psi_2(v))$.  So $\psi^{\s,H}(v)$ can be thought of as giving an $H$-configuration around $v$.  Unlike the definitions (in sections 1.2 and 1.6) for the infinite space this definition depends on a choice of sofic approximation $\s \in \Hom(G,\sym(n))$.  Notice that both $\phi^H$ and $\psi^{\s,H}$ take values in $\C^H$.  So we can define $d^H_\s(\phi,\psi)$ to be the $\ell^1$-distance between $\phi_*^H(\mu)$ and $\psi_*^{\s,H}\Unif_n$:
 $$d^H_\s(\phi,\psi) = \sum_{a\in \C^H} \left|\mu\left(\{x\in \A^G:~\phi^H(x)=a\}\right) - \Unif_n\left(\left\{i \in [n]:~ \psi^{\s,H}(i)=a\right\}\right)\right|.$$
 Here, $\Unif_n$ is the uniform probability measure on $[n]$. Note that $H_1 \subset H_2$ implies $d^{H_1}_\s(\phi,\psi) \le d^{H_2}_\s(\phi,\psi)$. Let $d^*_\s(\phi,\psi) = \sum_{i=1}^r d^{\{e,s_i\}}_\s(\phi,\psi)$.

By Theorem 1.4 in \cite{bowen-entropy-2010b}, 
$$F_\mu( \phi) = \inf_{\eps>0} \limsup_{n\to\infty} \frac{1}{n} \log \EE_{\sigma \sim u_n}  |\{\psi \in \C^n: d^*_\s (\phi,\psi) \leq \eps\}|.$$  
Define 
$$F^Y_\mu( \phi):= \inf_{\eps>0} \limsup_{n\to\infty} \frac{1}{n} \log \EE_{\sigma \sim u_n} |\{\psi \in \C^n: d^*_\s(\phi,\psi) \leq \eps, P^\s_{\psi}(Y) = 1\}|.$$  
Also note that $f^Y$ can equivalently be expressed as  
$$f^Y_\mu(\phi)=\inf_{K\subset G}\inf_{\eps>0} \lim\sup_n \frac{1}{n} \log \EE_{\sigma \sim u_n}  |\{\psi \in \C^n: d^K_\s(\phi,\psi) \leq \eps, P^\sigma_\psi(Y) = 1\}|$$
where the infimum is over finite $K \subset G$. 
\subsubsection{Markov chains}


We now introduce Markov processes in the same way as defined in \cite{bowen-entropy-2010a}.

Let $(X,\mu,T,G)$ be a pmp action.  Let $\phi: X \to \C$ be an observable. For $H \subset G$ finite recall that $\phi^H = \bigvee_{h\in H} \phi\circ T^{h^{-1}}$.  We also identify $\phi$ with the partition and $\s$-algebra it induces on $X$. Note that an element of the partition $\phi^H$ is of the form $\cap_{h\in H} T^{h^{-1}}A_h$, where each $A_h \in \phi$. 

Let the left Cayley graph $\G_L$ with respect to $(G,S)$ have vertex set $G$ and an edge between $g$ and $sg$ for each $g\in G$ and $s \in S$.  For $g_1,g_2\in G$ let $\Past(g_1;g_2)$ be the set of all $f\in G$ such that every path from $f$ to $g_1$ passes through $g_2$.  

If $\cF$ is a $\s$-algebra and $A \subset X$, then we write $\mu(A|\cF):X \to \R$ for the conditional expectation of the characteristic function $1_A$ conditioned on $\cF$. This is well-defined mod $\mu$. 

\begin{defn}
Let $\phi$ be a measurable partition of $X$. $(X,\mu,T,\phi)$ is a {\bf Markov process} if for every $g\in G$, $s\in S\cup S^{-1}$, $A \in \phi$, 
$$\mu\left(T^{(sg)^{-1}}A|\phi^{\Past(sg;g)}\right) = \mu\left(T^{(sg)^{-1}}A|\phi^g\right)=\mu\left(T^{s^{-1}}A|\phi\right).$$
\end{defn}

The second equality above holds for all $G$-invariant measures $\mu$.

We say that a Markov process $(X, \mu, T,\phi)$ has {\bf rational probabilities} if for every $a,b\in \C$, $s\in S$, $\mu(\phi^{-1}(a) \cap s\phi^{-1}(b)) \in \Q$.  For the remainder of the section we will specialize to Markov processes of symbolic dynamical systems, specifically $X=\A^G$.

\begin{lem}\label{lem : Markov}(Theorem 6.1 in \cite{bowen-entropy-2010a})
Let $(\A^G, \mu, \phi)$ be a Markov process.  Then $F_\mu(\phi) = f_\mu( \phi)$.
\end{lem}

\begin{lem}\label{lem : Markov simplification}
Let $(\A^G, \mu, \phi)$ be a Markov process with rational probabilities. Choose $\phi,\C, Y$ as in  \S \ref{S:restricted}.  Then $F_\mu(\phi) = F^Y_\mu(\phi) = f^Y_\mu(\phi).$
\end{lem}

We shall assume the above lemma and prove it later.  Let ${B_n} = B(e,n)$ be the ball of radius $n$ around the identity.  We also need the following lemma.

\begin{lem}\label{lem : Markov approx}
\begin{enumerate}
    \item Let $(\A^G,\mu,\phi)$ be a Markov process such that $\Phi$ is a homeomorphism and $\Phi_*\mu(Y) = 1$.  Then there exists a sequence $\mu_n \to \mu$ in weak* such that each $(\A^G,\mu_n,\phi)$ is Markov with rational probabilities and $\Phi_*\mu_n(Y) =1$.
    \item Let $(\A^G,\mu)$ be a symbolic dynamical system, $Z \subset \A^G$ a subshift of finite type with $\mu(Z) = 1$ and let $\phi:\A^G \to \A$ be the canonical observable (i.e. $\phi(x) = x_e$).  Then there exists a sequence $\{\mu_n\}_{n=1}^\infty$ of invariant Borel probability measures on $\A^G$ such that $\mu_n \to\mu$ in weak* as $n\to\infty$, $(\A^G,\mu_n,\phi^{B_n})$ is Markov, $\mu_n(Z)=1$, and $F_{\mu_n}(\phi^{B_n}) = F_{\mu}(\phi^{B_n})$ for all $n$. 
\end{enumerate}
\end{lem}

We will assume (1) of the above lemma and prove it later. Item
(2) of the lemma  above is proven in Section 9 of \cite{bowen-entropy-2010a}. 

We will also make use of the claim below. 
\begin{claim}\label{C:us}
The map $\mu \mapsto f^Y_\mu(\phi)$ is upper-semicontinuous in the following sense. If $\{\mu_n\}_{n=1}^\infty$ is a sequence of invariant Borel probability measures weak*-converging to $\mu$ and $\Phi_*\mu_n$ is supported on $Y$ for all $n$, then $f^Y_\mu(\phi) \ge \limsup_{n\to\infty} f^Y_{\mu_n}(\phi)$.
\end{claim}
\begin{proof}
Let $\nu = \Phi_*\mu$ and $\nu_n = \Phi_*\mu_n$.  Fix a weak* neighborhood $\cO$ of $\nu$. For all large enough $n$, $\nu_n \in \cO$.  For each such $n$, for all small enough weak* neighborhoods $\cO_n$ of $\nu_n$, $\cO_n \subset \cO$ so that $\Om_Y(\s,\cO_n) \subset \Om_Y(\s, \cO)$. It follows that  
$$\limsup_{m \to \infty} \frac{1}{m} \log \EE_{\sigma \sim u_m} \abs{\Omega_Y(\sigma, \cO)} \ge \limsup_{n\to\infty} f^Y_{\mu_n}(\A^G,\phi).$$
Take the infimum over $\cO$ to obtain $f^Y_\mu(\phi) \ge \limsup_{n\to\infty} f^Y_{\mu_n}(\phi)$.
\end{proof}



\begin{proof}[Proof of Theorem \ref{thm:SFT}]
Let $\phi:\A^G \to \A$ be the canonical observable $\phi(x)=x_e$. By item (2) of Lemma \ref{lem : Markov approx}, there exists a sequence $\{\mu_n\}_{n=1}^\infty$ of invariant Borel probability measures on $\A^G$ such that $\mu_n \to\mu$ in weak* as $n\to\infty$, $(\A^G,\mu_n,\phi^{B_n})$ is Markov, $\mu_n(Z)=1$, and $F_{\mu_n}(\phi^{B_n}) = F_{\mu}(\phi^{B_n})$ for all $n$. 

By item (1) of Lemma \ref{lem : Markov approx} applied to each $\mu_n$, there exist $\mu'_n$ such that $\mu'_n \to\mu$ (weak*) as $n\to\infty$, $(\A^G,\mu'_n,\phi^{B_n})$ is Markov with rational probabilities, $\mu'_n(Z)=1$, and $F_{\mu'_n}(\phi^{B_n}) \ge F_{\mu}(\phi^{B_n})-o(n)$. 

We claim
$$f^Z_{\mu'_n}(\phi) = f^Y_{\mu'_n}(\phi^{B_n})= F^Y_{\mu'_n}(\phi^{B_n}) = F_{\mu'_n}(\phi^{B_n}) \ge  F_{\mu}(\phi^{B_n}) - o(n)$$
where $Y=\Phi_n(Z)$ and $\Phi_n: \A^G \to (\A^{B(e,n)})^G$ is the equivariant map determined by $\phi^{B_n}$ (so $\Phi_n(x)(g)=\phi^{B_n}(g^{-1}x)$). The first equality holds by Claim \ref{C : restricted_invar} and the second and third equalities follow from Lemma \ref{lem : Markov simplification} (note that $Y$ is a nearest neighbor subshift of finite type when $n$ is large enough).  By applying $\limsup_{n\to\infty}$ to all parts of the equality and using the upper semi-continuity from Claim \ref{C:us}, we obtain $f^Z_\mu(\phi) \ge f_\mu(\phi)$ as desired.

\end{proof}

\begin{proof}[Proof of Lemma \ref{lem : Markov simplification}]
First we show that $F_\mu(\phi) = F^Y_\mu(\phi)$.  Because $F_\mu(\phi) \ge F^Y_\mu(\phi)$ is immediate from the definitions, it suffices to show $F_\mu(\phi) \le F^Y_\mu(\phi)$.
Let 
$$G_\mu(\phi) = \limsup_{n\to\infty} \frac{1}{n} \log\EE\left(\left|\left\{\psi \in \C^n: d^*_\s (\phi,\psi) =0 \right\}\right|\right) .$$
Now $F^Y_\mu(\phi) \ge G_\mu(\phi)$ because for any homomorphism $\s: G \to \sym(n)$, all $\psi \in \C^n$ such that $d^*_\s(\phi,\psi) = 0$ automatically satisfy $P^\s_\psi(Y) = 1$ whenever $Y$ is a nearest neighbor SFT. By Lemma 2.2 in \cite{bowen-entropy-2010b}, $G_\mu(\phi) \ge F_\mu(\phi)$, so $F_\mu(\phi) = F^Y_\mu(\phi)$.

Next we show that $F^Y_\mu(\phi) = f^Y_\mu(\phi)$. Given a finite subset $K \subset G$, let 
$$g_\mu(\phi,K) =  \limsup_{n\to\infty} \frac{1}{n} \log \EE\left(\left|\left\{\psi \in \C^n: d^K_\s (\phi,\psi) =0 \right\}\right|\right).$$
Let ${B_m} = B(e,m)$ be the ball of radius $m$ around the identity. We will use the claim below:
\begin{claim}
For every $m$, $g_\mu(\phi,{B_m}) \geq G_\mu(\phi^{B_m})$.
\end{claim}
\begin{proof}
Fix $\s$. By taking limits, it suffices to show for all $n$ large enough  
$$\left| \left\{\psi \in \C^n: d^{B_m}_\s (\phi,\psi) =0 \right\}\right| \geq \left|\left\{\eta \in (\C^{B_m})^n: d^*_\s (\phi^{B_m},\eta) =0 \right\}\right|.$$
It suffices to find an injective map from  $\{\eta \in (\C^{B_m})^n: d^*_\s (\phi^{B_m},\eta) =0 \}$ to $\{\psi \in \C^n: d^{B_m}_\s (\phi,\psi) =0 \}$.

Let $\eta \in (\C^{B_m})^n$ be such that $d^*_\s (\phi^{B_m},\eta) =0$.  Define $\psi \in \C^n$ by $\psi(v)=\eta(v)(e)$. Note that $d_\s^{B_m}(\phi,\psi) = 0$. We claim that the map $\eta \mapsto \psi$ is injective.

Observe that $d^*_\s(\phi^{B_m},\eta) = 0$ implies that for any $s\in S$, $g \in B(e,m) \cap B(s,m)$, $\eta(v)(g) = \eta(\s(s^{-1})v)(s^{-1}g)$.  In particular, let $f \in B(e,m)$ and write $f = s_1s_2\cdots s_l$.  Then 
\begin{eqnarray*}
\eta(v)(f) &=& \eta(v)(s_1\cdots s_l) = \eta(\s(s_1^{-1})v)(s_2\cdots s_l) \\
&=& \cdots = \eta(\s(s_l^{-1}\cdots s_1^{-1})v)(e)= \psi(\s(f^{-1})v).
\end{eqnarray*}
This shows $\eta$ can be recovered from $\psi$. So $\eta \mapsto \psi$  is injective.

\end{proof}
Let 
$$f^Y_\mu(\phi,K) = \inf_{\eps>0} \limsup_{n\to\infty} \frac{1}{n} \log\EE\left(\left|\left\{\psi  \in \C^n: d^K_\s (\phi,\psi) \le \eps, P^\s_\psi(Y) = 1\right\}\right|\right).$$
By an argument similar to why $F^Y_\mu(\phi) \ge G_\mu(\phi)$, we obtain $f^Y_\mu(\phi,K) \ge g_\mu(\phi,K)$ for any finite $K \subset G$ containing $S \cup S^{-1} \cup \{e\}$.  Thus we have for every $m\in \NN$,
$$f^Y_\mu(\phi,{B_m}) \ge g_\mu(\phi,{B_m}) \ge G_\mu(\phi^{B_m}) \ge F_\mu(\phi^{B_m}) = f_\mu(\phi^{B_m}) = f_\mu(\phi) = F_\mu(\phi) = F^Y_\mu(\phi)$$
noting that $(A^G,\mu,\phi^{B_m})$ is also Markov by Lemma 6.3 in \cite{bowen-entropy-2010a}. The result follows by applying $\inf_m$ to all parts of the above (and $f^Y_\mu(\phi) \le F^Y_\mu(\phi)$ by definition).
\end{proof}


\begin{proof}[Proof of Lemma \ref{lem : Markov approx}(1)]

We need an analogue of Lemma 2.3 from \cite{bowen-entropy-2010b}. To motivate it, suppose that $\nu$ is an invariant probability measure on $\A^G$. For $a,b \in \A$ and $1\le i \le r$, define
$$W_\nu(a) = \nu(\{ z \in \A^G:~ z(e) = a\})$$
$$W_\nu(a,b;i) = \nu(\{z \in \A^G:~z(e)=a \textrm{ and } z(s_i) = b \}).$$
Then $W_\nu$ is a weight. Formally, a {\bf weight} is a function
$$W: \A \sqcup (\A \times \A \times [r]) \to [0,1]$$
satisfying the conditions below.
\begin{enumerate}
\item  (Balanced). For every $i$, and every $a \in \A$, $W(a) = \sum_{b\in \A} W(a,b;i)  = \sum_{b \in \A} W(b,a;i).$
    \item (Normalized).  $\sum_a W(a) = 1$.
\end{enumerate}

Given two weights $W_1$,$W_2$ define $d(W_1,W_2) = \sum_{(a,b;i) \in \A \times \A \times [r]}|W_1(a,b;i)-W_2(a,b;i)|$.

A weight $W$ is a {\bf $Y$-weight} if for every $(a,b;i) \in \A \times \A \times [r]$, if there does not exist $x\in X$ with $\Phi(x)\in Y$, $x(e)=a$ and $x(s_i)=b$ then $W(a,b;i)=0$. If the measure $\nu$ is supported on $\Phi^{-1}(Y)$ then $W_\nu$ is a $Y$-weight. 

A $Y$-weight as defined above induces an invariant transition system as defined in section 7 in \cite{bowen-entropy-2010a}, and there it is shown that an invariant transition system induces a Markov process. In particular, for every $Y$-weight $W$, there is a unique Markov measure $\nu$ such that $W_\nu=W$.

The equations defining $Y$-weights are linear equations with rational coefficients. Therefore, the subset of  rational-valued $Y$-weights is dense (with respect to the distance function defined above) in the space of all $Y$-weights. It is straightforward to check that convergence of a sequence of weights implies convergence in weak* of the associated Markov measures and that a rational-valued $Y$-weight corresponds to a Markov measure with rational probabilities. So this implies the lemma.
\end{proof}

\section{Standard hypotheses and notation}\label{S:prelim}

\begin{defn}
The {\bf standard hypotheses} for this paper are the following: $(X,\mu)$ is a standard probability space, $G=\langle S \rangle$ is a free group with free generating set $S=\{s_1,\ldots s_r\}$, $T=(T^g)_{g\in G}$, $U=(U^g)_{g\in G}$ are essentially free pmp actions of $G$ on $(X,\mu)$ with the same orbits. Define cocycles $\a: G\times X \to G$, $\beta:G \times X \to G$ by
$$U^{\a(g,x)}x = T^gx, \quad T^{\beta(g,x)}x = U^g x.$$

Note that the above definition gives 
\begin{eqnarray}\label{E:inv1}
 \a(\b (g,x),x) = g= \b(\a(g,x),x).
\end{eqnarray}
We assume there exist a finite set $\B$ and a measurable map $\g:X \to \B$ such that $\g$ is both $T$-generating and $U$-generating. We do not make any quantitive orbit-equivalence assumptions until the end of the paper. 

\end{defn}

\section{The space of orbit-change maps}\label{S:orbit-change}

 For $x\in X$, define the {\bf orbit-change maps} $\hat{\alpha}_x:G \to G$ and $\hat{\beta}_x: G \to G$ by 
$$\hat{\alpha}_x(g)=\a(g^{-1},x)^{-1}, \quad \hat{\beta}_x(g)=\beta(g^{-1},x)^{-1}.$$
Then $\hat{\alpha}_x, \hat{\beta}_x\in \sym_e(G)$ where $\sym_e(G)$ is the set of bijections $\phi:G \to G$ such that $\phi(e)=e$. These maps satisfy the multiplication rules
$$\hat{\a}_{T^gx}(gh) = \a(g,x) \hat{\a}_x(h), \quad \hat{\beta}_{U^gx}(gh) = {\beta}(g,x) \hat{\beta}_x(h).$$

Also note that the above definitions and (\ref{E:inv1}) gives
\begin{eqnarray}\label{E:inv2}
\hal_x^{-1}=\hbe_x.
\end{eqnarray}

Note $\sym_e(G)$ is a group under composition. Moreover, it is a Polish group with respect to the pointwise convergence topology. In fact, if $\{g_i\}_{i=1}^\infty$ is an enumeration of $G$ then
$$d(\phi,\psi) = \sum_{i=1}^\infty 2^{-i} (1_{\phi(g_i)=\psi(g_i)} + 1_{\phi^{-1}(g_i)=\psi^{-1}(g_i)})$$
is a complete separable metric inducing the pointwise convergence topology.

We would like to say that the maps $x \mapsto \hat{\alpha}_x$ and $x \mapsto \hat{\beta}_x$ are equivariant. So we define actions $\Theta, \mho$ of $G$ on $\sym_e(G)$ by
\begin{eqnarray*}
(\Theta^h\phi)(g) &=& \phi(h^{-1})^{-1}\phi(h^{-1}g),\\
(\mho^h\phi)(g) &=& h\phi(\phi^{-1}(h^{-1})g) = \left(\Theta^{\phi^{-1}(h^{-1})^{-1}}\phi\right)(g).
\end{eqnarray*}
These two actions of $G$ have the same orbits.

Because of the multiplication rules, the map $x \mapsto \hat{\a}_x$ is $(T,\Theta)$-equivariant in the sense that 
\begin{eqnarray}\label{E:ha1} 
\hat{\a}_{T^gx}=\Theta^g\hat{\a}_x. 
\end{eqnarray}
It is also $(U,\mho)$-equivariant:
\begin{eqnarray}\label{E:ha2} 
\hat{\a}_{U^gx}=\mho^g\hat{\a}_x.
\end{eqnarray}
For example (\ref{E:ha2}) follows from (\ref{E:ha1}) and the observation (which follows from (\ref{E:inv2})) that
\begin{eqnarray}\label{E:inv3}
 \hal_x^{-1}(h^{-1})^{-1} = \b(h,x).
\end{eqnarray}
Similarly, the map $x \mapsto \hat{\beta}_x$ is $(T,\mho)$- and $(U,\Theta)$-equivariant:
$$\hat{\beta}_{T^gx}=\mho^g\hat{\beta}_x, \quad  \hat{\beta}_{U^gx}=\Theta^g\hat{\beta}_x.$$
These properties follow from the cocycle identities (\ref{E:cocycle1}), (\ref{E:cocycle2}).

\subsection{A partially symbolic model}\label{sec:Gamma}
Recall that $\g:X \to \B$ is generating for the $U$ and $T$ actions. Recall symbolic dynamics from section 2: $\B^G$ is the space of all functions $y:G \to \B$ with the topology of pointwise convergence on finite sets, for $x \in \B^G$ we will use either function notation or subscripts (so $y(g)=y_g$) whichever is most convenient, and $G$ acts on $\B^G$ by the left shift action $(gy)(f)=y(g^{-1}f)$. 

Define $\G: X \to \B^G$  by
 $$\G(x)_g = \g(T^{g^{-1}}x).$$
 This map is equivariant in the sense that $\G(T^h x) = h \G(x)$ for all $h \in G$, $x\in X$. Because $\g$ is $T$-generating, this map is also 1-1 (modulo null sets).
 
Define $\tGamma: X \to \sym_e(G) \times \B^G$ by $\tGamma(x) = (\hat{\alpha}_x, \G(x))$. Also define actions $\tTheta$, $\tilde{\mho}$ of $G$ on $\sym_e(G) \times \B^G$ by
\begin{eqnarray}\label{E:defn}
\tTheta^g(\phi,y) = (\Theta^g\phi, gy), \quad \tilde{\mho}^g (\phi, y)  = (\mho^g \phi, \phi^{-1}(g^{-1})^{-1}y).
\end{eqnarray}
By (\ref{E:inv3}), (\ref{E:ha1}), and (\ref{E:ha2}), $\tGamma$ is doubly-equivariant in the sense that
$$\tGamma(T^g x) = \tTheta^g \tGamma(x), \quad \tGamma(U^g x) = \tilde{\mho}^g \tGamma(x).$$
Because $\g$ is both $T$ and $U$ generating, $\tGamma$ is injective (modulo null sets). So to prove Theorem \ref{thm:main}, it suffices to prove $f_{\tGamma_*\mu}(\tTheta) = f_{\tGamma_*\mu}( \tilde{\mho})$.

\section{A subshift of finite type for bounded orbit-equivalences}\label{S:SFTBOE}

The goal of this section is to show that there is a subshift of finite type which encodes bounded orbit-change maps. To begin, let $\A = G^{S \cup S^{-1}}$ where $S=\{s_1,\ldots, s_r\}$ is the generating set. If $x \in \A^G$, $g\in G$ and $s\in S \cup S^{-1}$ then we write $x_g(s) = x(g)(s) \in G$ to simplify notation.


Define $\cE:\sym_e(G) \to \A^G$ by 
$$\cE(\phi)_h(s) = \Theta^{h^{-1}}\phi(s)=\phi(h)^{-1}\phi(hs).$$
This map is an embedding in the sense that it is equivariant, continuous and injective. The equivariance means that 
\begin{eqnarray}\label{E:equi1}
\cE(\Theta^h \phi) = h\cE(\phi)
\end{eqnarray}
for $h \in G$, $\phi \in \sym_e(G)$. In fact $\cE$ is determined by this equivariance condition and the formula $\cE(\phi)_{e}(s)=\phi(s)$.

If $x = \cE(\phi)$ and $s_1,\ldots, s_n \in S\cup S^{-1}$ then
\begin{eqnarray}\label{E:inverse}
\phi(s_1\cdots s_n) = x_{e}(s_1)x_{s_1}(s_2)x_{s_1s_2}(s_3)\cdots x_{s_1\cdots s_{n-1}}(s_n).
\end{eqnarray}
This is obtained via induction on $n$. This verifies that $\cE$ is injective. We have the following more general fact:
\begin{lem}\label{L:inverse}
If $\phi \in \sym_e(G)$, $x = \cE(\phi)$, $g\in G$ and $s_1,\ldots, s_n \in S\cup S^{-1}$ then
\begin{eqnarray*}
\phi(gs_1\cdots s_n)=\phi(g)x_g(s_1)x_{gs_1}(s_2)\cdots x_{gs_1\cdots s_{n-1}}(s_n).
\end{eqnarray*}
Moreover, for any $t \in G$ and $t_1,\ldots, t_m \in S \cup S^{-1}$, the equation
\begin{eqnarray*}
\phi^{-1}(gt)=\phi^{-1}(g)t_1\cdots t_m
\end{eqnarray*}
is true if and only if 
$$t= x_{\phi^{-1}(g)}(t_1)\cdots x_{\phi^{-1}(g)t_1\cdots t_{m-1}}(t_m).$$
Moreover, such a sequence is unique if we require $m=|t_1\cdots t_m|_G$.
\end{lem}

\begin{proof}
The first statement follows from (\ref{E:inverse}) by writing $g\in G$ as a word in $S\cup S^{-1}$. 

To prove the second statement, suppose $\phi^{-1}(gt)=\phi^{-1}(g)t_1\cdots t_m$. Apply $\phi$ to both sides, then apply the first statement to obtain
\begin{eqnarray*}
gt &=& \phi( \phi^{-1}(g)t_1\cdots t_m )\\
    &=& g x_{\phi^{-1}(g)}(t_1)x_{\phi^{-1}(g)t_1}(t_2)\cdots x_{\phi^{-1}(g)t_1\cdots t_{m-1}}(t_m).
\end{eqnarray*}
After cancelling $g$ from both sides, we obtain $t= x_{\phi^{-1}(g)}(t_1)x_{\phi^{-1}(g)t_1}(t_2)\cdots x_{\phi^{-1}(g)t_1\cdots t_{m-1}}(t_m)$ as required. The converse is obtained by following the same steps in reverse. Moreover, $t_1,\ldots, t_m \in S \cup S^{-1}$ are uniquely determined by two conditions: $m=|t_1\cdots t_m|$ and 
$$t_1\cdots t_m = \phi^{-1}(g)^{-1}\phi^{-1}(gt).$$
\end{proof}

For $\rho \in \NN$, let $\A_\rho = \ball{e}{\rho}^{S \cup S^{-1}}$ which we view as a subset of $\A$. Because $\A_\rho$ is finite, $\A_\rho^G$ is compact. Also, let $\sym_\rho(G) \subset \sym_e(G)$ be the subset of $\phi \in \sym_e(G)$ such that
$$|(\Theta^{g^{-1}}\phi)(s)|_G = |\phi(g)^{-1}\phi(gs)|_G \le \rho~ \textrm{ and } ~|(\Theta^{g^{-1}}\phi^{-1})(s)|_G = |\phi^{-1}(g)^{-1}\phi^{-1}(gs)|_G \le \rho$$
for all $s\in S\cup S^{-1}$ and $g\in G$.  In other words, if $x=\cE(\phi)$ and $y=\cE(\phi^{-1})$ then 
$|x_g(s)|_G \le \rho$ and $|y_g(s)|_G \le \rho$. In particular, if $\phi \in \sym_\rho(G)$, then $\cE(\phi) \in \A_\rho^G$.


The main theorem of this section is:
\begin{thm}\label{T:SFT1}
For $\rho \in \NN$, $\cE(\sym_\rho(G)) \subset \A_\rho^G$ is a subshift of finite type.
 \end{thm}
 
 \begin{proof}
Let $F =  \ball{e}{\rho^2 + 1}\subset G$ be the radius $\rho^2+1$ ball. Let $\cV \subset \A_\rho^F$ be the set of all maps $z:F \to \A_\rho$ such that
\begin{enumerate}
\item[(Axiom 1)] $z_{e}(s)z_{s}(s^{-1})=e$ for all $s\in S \cup S^{-1}$.
\item[(Axiom 2)] For every $h\in G$ with $|h|_G \le \rho$ there exists a unique sequence $s_1,\ldots, s_n \in S\cup S^{-1}$ with $|s_1\cdots s_n|_G = n \le \rho^2+1$ such that
$$h=z_{e}(s_1)z_{s_1}(s_2)z_{s_1s_2}(s_3)\cdots z_{s_1\cdots s_{n-1}}(s_n).$$
Moreover, $n\le \rho|h|_G$.
\end{enumerate}
Let $\cW = \A_\rho^F \setminus \cV$ be the complement of $\cV$. Let $Z$ be the subshift of finite type determined by $\cW$. This means $Z$ is the set of all $x\in \A^G$ such that for every $g\in G$,  $gx$ restricted to $F$ is in $\cV$. We will show $Z =  \cE(\sym_\rho(G))$.

Let $x\in Z$. We claim that
\begin{eqnarray}\label{E:reverse}
x_g(s)x_{gs}(s^{-1})=e
\end{eqnarray}
for all $s\in S\cup S^{-1}$ and $g\in G$. By Axiom 1 applied to $z=g^{-1}x$, 
$$(g^{-1}x)_{e}(s)(g^{-1}x)_{s}(s^{-1})=e.$$
Equation (\ref{E:reverse}) follows from this and the fact that $(g^{-1}x)_{e}(s) = x_g(s)$, $(g^{-1}x)_{s}(s^{-1})=x_{gs}(s^{-1})$ by definition of the $G$-action on $\A^G$. 

\noindent {\bf Claim 1}. $Z \supset \cE(\sym_\rho(G))$.

\begin{proof}[Proof of Claim 1]

Let $\phi \in \sym_\rho(G)$ and suppose $x=\cE(\phi) \in \A_\rho^G$. We will show $x\in Z$. If $s\in S \cup S^{-1}$ then
$$x_{e}(s) x_s(s^{-1}) =\phi(e)^{-1} \phi(s)\phi(s)^{-1}\phi(s^{-1}s) = e.$$
This verifies Axiom 1. 

Now let $h \in G$ with $|h|_G \le \rho$.  Let $h=h_1\cdots h_k$ for $h_i \in S \cup S^{-1}$ and $|h|_G=k$. Then
$$\phi^{-1}(h)  = \big(\phi^{-1}(e)^{-1}\phi^{-1}(h_1)\big)\big(\phi^{-1}(h_1)^{-1}\phi^{-1}(h_1h_2)\big) \cdots \big(\phi^{-1}(h_1\cdots h_{k-1})^{-1}\phi^{-1}(h_1h_2\cdots h_k)\big).$$
Since each term in parenthesis has word length bounded by $\rho$, this shows $|\phi^{-1}(h)|_G \le |h|_G \rho\le \rho^2$.

By Lemma \ref{L:inverse} applied to $g=e$, $t=h$, $\phi^{-1}(h)=s_1\cdots s_n$ where $s_1,\ldots, s_n \in S \cup S^{-1}$ is uniquely determined by $|s_1\cdots s_n|_G = n$ and
$$h=x_{e}(s_1)\cdots x_{s_1\cdots s_{n-1}}(s_n).$$
This verifies Axiom 2. Therefore, $x\in Z$. Since $\phi$ is arbitrary, this proves $Z \supset \cE(\sym_\rho(G))$.
\end{proof}

\begin{lem}
$Z \subset \cE(\sym_\rho(G))$.
\end{lem}
\begin{proof}

Fix $x\in Z$.  We will define some map $\phi: G \to G$ and show 
\begin{enumerate}
    \item $\phi \in \sym_\rho(G)$.
    \item $\cE(\phi) = x$. 
\end{enumerate}
Define $\phi:G \to G$ as follows. First let $\phi(e)=e$. Assuming $\phi(g)$ has been defined for some $g\in G$, define $\phi(gs)$ (for $s\in S\cup S^{-1}$ with $|gs| = |g|+1$) by
\begin{eqnarray}\label{E:increment}
\phi(gs) = \phi(g)x_g(s).
\end{eqnarray}
This uniquely defines $\phi$. It will be convenient to know that (\ref{E:increment}) holds even without the assumption $|gs|=|g|+1$. So suppose $|gs|=|g|-1$. By (\ref{E:increment}) applied to $gs$ and $s^{-1}$ (note that $|gss^{-1}| = |g| = |gs| + 1$),
\begin{eqnarray}\label{E:increment2}
\phi(g)= \phi(gss^{-1}) = \phi(gs)x_{gs}(s^{-1}).
\end{eqnarray}
By Axiom 1 applied to $z= (gs)^{-1}x$ and $s^{-1}$,
$$e=z_e(s^{-1})z_{s^{-1}}(s) = x_{gs}(s^{-1})x_g(s).$$
Thus $x_{gs}(s^{-1}) = x_g(s)^{-1}$. Substitute this into (\ref{E:increment2}) to obtain $\phi(gs) = \phi(g) x_g(s)$ as claimed. Thus assuming $\phi \in \sym_\rho(G)$ we can calculate that for any $h\in G$, $s \in S\cup S^{-1}$,  $\cE(\phi)_h(s) = \phi(h)^{-1}\phi(hs) = x_h(s)$, so $\cE(\phi) = x$.  So it remains to show that $\phi \in \sym_\rho(G)$.  

Also, by induction, (\ref{E:inverse}) holds for the $\phi$ that we have constructed.

 
 Our next goal is to prove that $\phi$ is surjective and injective, hence in $\sym_e(G)$, and then we will show $\phi \in \sym_\rho(G)$ and therefore $\cE(\sym_\rho(G))=Z$, which finishes the proof.

Note that $\phi(s) =x_{e}(s) \in \ball{e}{\rho}$ for $s\in S \cup S^{-1}$. So $|\phi(s)|_G \le \rho$ for all $s\in S \cup S^{-1}$. Moreover, $\phi(g)^{-1}\phi(gs)=x_g(s)$. So $| \phi(g)^{-1}\phi(gs)|_G \le \rho$ for all $g$.


\noindent {\bf Claim 2}. $\phi$ is surjective.

\begin{proof}[Proof of Claim 2]

By induction, it suffices to prove the following statement: for every $g \in G$ and $t\in S\cup S^{-1}$, if $g$ is in the image of $\phi$ then $gt$ is also in the image of $\phi$. So suppose $\phi(h)=g$ for some $h$. We apply Axiom 2 to $h^{-1}x$ to obtain the existence of $s_1,\ldots, s_n \in S\cup S^{-1}$ with $n \le \rho$ such that 
$$t=(h^{-1}x)_{e}(s_1)(h^{-1}x)_{s_1}(s_2)(h^{-1}x)_{s_1s_2}(s_3)\cdots (h^{-1}x)_{s_1\cdots s_{n-1}}(s_n).$$
Using the action of $G$ of $\A^G$, this implies
$$t=x_{h}(s_1)x_{hs_1}(s_2)x_{hs_1s_2}(s_3)\cdots x_{hs_1\cdots s_{n-1}}(s_n).$$
By (\ref{E:inverse}),
$$\phi(hs_1\cdots s_n) = \phi(h)x_{h}(s_1)x_{hs_1}(s_2)x_{hs_1s_2}(s_3)\cdots x_{hs_1\cdots s_{n-1}}(s_n)=gt.$$
This proves $\phi$ is surjective.
\end{proof}

We will now show $\phi$ is injective by demonstrating that it has an inverse $\psi$. Define $\psi:G \to G$ by
\begin{enumerate}
\item $\psi(e)=e$,
\item if $\psi(g)$ has been defined, $t \in S\cup S^{-1}$ and $|gt|_G = |g|_G + 1$ then define
\begin{eqnarray}\label{E:thing}
\psi(gt)=\psi(g)s_1\cdots s_n
\end{eqnarray}
where $s_1,\ldots, s_n \in S \cup S^{-1}$ is the unique sequence satisfying $n=|s_1\cdots s_n| \le \rho$ and 
$$t=(\psi(g)^{-1}x)_{e}(s_1)\cdots (\psi(g)^{-1}x)_{s_1\cdots s_{n-1}}(s_n) = x_{\psi(g)}(s_1)\cdots x_{\psi(g)s_1\cdots s_{n-1}}(s_n).$$
The existence and uniqueness of this sequence is guaranteed by Axiom 2 of the definition of $Z$.
\end{enumerate}

\noindent {\bf Claim 3}. Equation (\ref{E:thing}) holds for all $t\in S \cup S^{-1}$ (even if $|gt|_G\ne |g|_G+1$). 

\begin{proof}[Proof of Claim 3]
To see this, suppose $|gt| = |g|-1$ and let $h=gt$.  Then $|g| = |gt|+1$ or equivalently $|ht^{-1}| = |h|+1$ so by definition
$$\psi(g)=\psi(ht^{-1}) = \psi(h)t_1\cdots t_n$$
where $t_1,\ldots, t_n \in S \cup S^{-1}$ is the unique sequence satisfying 
\begin{eqnarray}\label{E:thing-1}
t^{-1}=x_{\psi(h)}(t_1)\cdots x_{\psi(h)t_1\cdots t_{n-1}}(t_n)
\end{eqnarray}
and $|t_1\cdots t_n|_G = n\le \rho$. Thus
\begin{eqnarray}\label{E:thing-2}
\psi(gt)=\psi(h) = \psi(g)t_n^{-1}\cdots t_1^{-1}.
\end{eqnarray}
By (\ref{E:thing-1}), (\ref{E:reverse}) and (\ref{E:thing-2})
\begin{eqnarray*}
t &=& x_{\psi(h)t_1\cdots t_{n-1}}(t_n)^{-1}\cdots x_{\psi(h)}(t_1)^{-1} \\
&=& x_{\psi(h)t_1\cdots t_n}(t_n^{-1})\cdots x_{\psi(h)t_1}(t_1^{-1}) \\
&=& x_{\psi(g)}(t_n^{-1})\cdots x_{\psi(g)t_n^{-1}\cdots t_{2}^{-1}}(t_1^{-1}).
\end{eqnarray*}
This proves the claim with $(s_1,\ldots, s_n) = (t_n^{-1},\ldots, t_1^{-1})$. 
\end{proof}

Let $g \in G$ and $t \in G$ with $|t|_G\le \rho$. By Claim 3 and induction on $|t|_G$, we obtain
\begin{eqnarray}\label{E:thing124}
\psi(gt)=\psi(g)s_1\cdots s_n
\end{eqnarray}
where $s_1,\ldots, s_n \in S \cup S^{-1}$ is the unique sequence satisfying $n=|s_1\cdots s_n| \le \rho^2$ and 
$$t=(\psi(g)^{-1}x)_{e}(s_1)\cdots (\psi(g)^{-1}x)_{s_1\cdots s_{n-1}}(s_n) = x_{\psi(g)}(s_1)\cdots x_{\psi(g)s_1\cdots s_{n-1}}(s_n).$$

\noindent {\bf Claim 4}. $\phi$ is injective.

\begin{proof}[Proof of Claim 4]
It suffices to prove $\psi (\phi(g))=g$ for all $g\in G$. This is true for $g=e$. By induction, it suffices to assume $\psi(\phi(g))=g$ and prove $\psi(\phi(gs))=gs$ for $s\in S \cup S^{-1}$. 

Let 
$$t=x_g(s) = \phi(g)^{-1}\phi(gs)$$ 
where the second equality holds by (\ref{E:inverse}) or (\ref{E:increment}). Since $x \in \A_\rho^G$, this implies 
$|t|_G \le \rho$.  By (\ref{E:thing124}), we have
 $$\psi(\phi(gs)) = \psi(\phi(g)t) = \psi(\phi(g))s_1\cdots s_n = gs_1\cdots s_n $$
where $s_1,\ldots, s_n \in S \cup S^{-1}$ is the unique sequence satisfying
$$t=x_{g}(s_1)\cdots x_{gs_1\cdots s_{n-1}}(s_n)$$
and $|s_1\cdots s_n| = n\le \rho^2$.  But we also have $t = x_g(s)$ so by uniqueness, $s = s_1\cdots s_n$ and $n =1$. So
$$\psi(\phi(gs)) = \psi(\phi(g)t) = gs_1\cdots s_n = gs.$$
as required.
\end{proof}

Recall from the definition of $\psi$ that for any $s\in S \cup S^{-1}$, $g \in G$,
$$\psi(gs) = \psi(g)s_1\cdots s_n$$
where $s_1,\ldots, s_n \in S \cup S^{-1}$ is the unique sequence satisfying $n=|s_1\cdots s_n| \le \rho$ and 
$$s=(\psi(g)^{-1}x)_{e}(s_1)\cdots (\psi(g)^{-1}x)_{s_1\cdots s_{n-1}}(s_n) = x_{\psi(g)}(s_1)\cdots x_{\psi(g)s_1\cdots s_{n-1}}(s_n).$$
In particular, since $\psi = \phi^{-1}$ this shows $|\phi^{-1}(s)|_G \le \rho$ for all $s \in S \cup S^{-1}$. Similarly, $|\phi^{-1}(g)^{-1}\phi^{-1}(gs)|_G \le \rho$ for any $g\in G$. Thus $\phi \in \sym_\rho(G)$.

Because $\cE(\phi)=x$ and $x\in Z$ is arbitrary, we must have $Z \subset \cE(\sym_\rho(G))$. 

\end{proof}
Since we have already shown the opposite inclusion, the two sets are equal. This proves the theorem.

 \end{proof}

 \subsection{Rearranging periodic orbits}\label{sec:finitemodels}
 
 Define $\cF:\sym_e(G) \to \A^G$ by 
$$\cF(\phi)(h)(s) =  \mho^{h^{-1}}\phi(s)=h^{-1}\phi(\phi^{-1}(h)s).$$
Then $\cF$ is an embedding and satisfies the equivariance condition 
\begin{eqnarray}\label{E:equi2}
\cF(\mho^h \phi) = h\cF(\phi)
\end{eqnarray}
for $h \in G$, $\phi \in \sym_e(G)$. 

The next lemma enables us to re-arrange the orbit of a periodic point in $\A_\rho^G$. If the original periodic point is approximately equidistributed with respect to an invariant measure $\zeta$ on $\A_\rho^G$ then the new periodic point will be approximately equidistributed with respect to $(\cF\circ \cE^{-1})_*\zeta$. This will help us map periodic points which witness the entropy of $(X,\mu,T,G)$ to periodic points witnessing the entropy of $(X,\mu,U,G)$.  





\begin{lem}\label{L:tau}
Let $\s: G \to \sym(n)$ be a homomorphism and $x \in \A^n$. Recall the pullback name $x^\s_v$ from \S \ref{S:symbolic}. Suppose $x^\s_v \in Z=\cE(\sym_\rho(G))$ for all $v \in [n]$ (and some $\rho>0$). For $v \in [n]$, let  $\phi_v = \cE^{-1}(x^\s_v) \in \sym_e(G)$. Then there exists a unique homomorphism $\tau:G \to \sym(n)$  satisfying
$$\tau(g)v = \s(\phi_v^{-1}(g^{-1})^{-1})v$$
for all $g\in G$ and $v\in [n]$. Moreover, 
$$\cF(\cE^{-1}(x^\s_v)) = x^\tau_v$$
for all $v\in [n]$.
\end{lem}

\begin{proof}
To prove that $\tau$ is a homomorphism, it suffices to prove that for any $t, g \in G$, $\tau(t^{-1})\tau(g^{-1})v=\tau(t^{-1}g^{-1})v$. This is implied by
\begin{eqnarray}\label{E:1}
\phi^{-1}_{w}(t)^{-1} \phi_v^{-1}(g)^{-1} = \phi_v^{-1}(gt)^{-1}.
\end{eqnarray}
where $w = \s(\phi^{-1}_v(g)^{-1})v$. Choose $t_1,\ldots, t_m \in S \cup S^{-1}$ so that
$$ \phi^{-1}_{w}(t)=t_1\cdots t_m.$$
Apply $ \phi_{w}$ to both sides, then (\ref{E:inverse}) to obtain
\begin{eqnarray*}
t =  \phi_{w}(t_1\cdots t_m) = (x^\s_{w})_e(t_1)(x^\s_{w})_{t_1}(t_2)\cdots (x^\s_{w})_{t_1\cdots t_{m-1}}(t_m).
\end{eqnarray*}
For any $1\le i \le m$ we follow the definitions of $x^\s_w$, $w$ and $x^\s_v$ to obtain
\begin{eqnarray*}
(x^\s_w)_{t_1\cdots t_{i-1}}(t_i) &=& x(\s(t_{i-1}^{-1}\cdots t_1^{-1})w)(t_i) = x(\s(t_{i-1}^{-1}\cdots t_1^{-1}\phi^{-1}_v(g)^{-1})v)(t_i)\\
& =& (x^\s_v)_{\phi^{-1}_v(g)t_1\cdots t_{i-1}}(t_i).
\end{eqnarray*}
Combined with the previous formula, this gives
$$t =  (x^\s_v)_{\phi^{-1}_v(g)}(t_1)(x^\s_v)_{\phi^{-1}_v(g)t_1}(t_2)\cdots (x^\s_v)_{\phi^{-1}_v(g)t_1\cdots t_{m-1}}(t_m).$$
Now we apply Lemma \ref{L:inverse} with $\phi_v$ in place of $\phi$ and $x^\s_v$ in place of $x$ to obtain
$$ \phi_v^{-1}(gt) = \phi_v^{-1}(g)t_1\cdots t_m = \phi_v^{-1}(g) \phi^{-1}_{w}(t).$$
Take inverses to obtain (\ref{E:1}). This shows that $\tau$ is a homomorphism.

To prove the last statement, we check
$$\cF(\phi_v)_g(s) = g^{-1}\phi_v(\phi_v^{-1}(g)s) $$
by definition of $\cF$. Since $\cE(\phi_v)=x^\s_v$, the definition of $\cE$ gives
$$(x^\s_v)_{\phi_v^{-1}(g)}(s) = g^{-1}\phi_v(\phi_v^{-1}(g)s).$$
So $\cF(\phi_v)_g(s) = (x^\s_v)_{\phi_v^{-1}(g)}(s)$. By definition of $x^\s_v$, $\tau$ and $x^\tau_v$, we have
$$(x^\s_v)_{\phi_v^{-1}(g)}(s) = x(\s(\phi_v^{-1}(g)^{-1})v)(s)= x(\tau(g^{-1})v)(s) = (x^\tau_v)_g(s).$$
Since $g$ and $s$ are arbitrary, $\cF(\phi_v) = x^\tau_v$. Because $\phi_v = \cE^{-1}(x^\s_v)$, this implies the last claim.

\end{proof}

Next we extend $\cE$ and $\cF$ as follows. Define $\tcE, \tcF:\sym_e(G) \times \B^G \to \A^G \times \B^G$ by
$$\tcE(\phi,y)=(\cE(\phi),y), \quad \tcF(\phi,y) = (\cF(\phi), y \circ \phi^{-1}).$$
We claim these maps have the following equivariance properties:
\begin{eqnarray}\label{E:tcE}
\tcE\circ \tTheta^h = h \tcE, \quad \tcF \circ \tilde{\mho}^h = h \tcF.
\end{eqnarray}
The first equality above is straightforward and left to the reader (use (\ref{E:equi1}), (\ref{E:defn})). Verifying the second equality above is also straightforward but a little long. We will use the notation $y \circ h$ for the function which takes $g \in G$ to $y(hg)$. For any $(\phi,y) \in \A^G \times \B^G$, the definition of $\tilde{\mho}$ and $\tcF$ give
\begin{eqnarray*}
\tcF \circ \tilde{\mho}^h(\phi,y) &=& \tcF(\mho^h\phi, \phi^{-1}(h^{-1})^{-1}y) = \tcF(\mho^h\phi, y \circ \phi^{-1}(h^{-1})) \\
&=& (\cF\mho^h\phi, y \circ \phi^{-1}(h^{-1}) \circ (\mho^h \phi)^{-1}).
\end{eqnarray*}
On the other hand,
 \begin{eqnarray*}
 h \tcF(\phi,y) &=& h(\cF \phi, y \circ \phi^{-1}) = (h\cF \phi, y \circ \phi^{-1}\circ h^{-1}).
\end{eqnarray*}
By (\ref{E:equi2}), the first coordinates are equal. So it now suffices to show 
$$y \circ \phi^{-1}(h^{-1}) \circ (\mho^h \phi)^{-1} = y \circ \phi^{-1}\circ h^{-1}.$$
Equivalently, for all $g\in G$,
$$y (\phi^{-1}(h^{-1}) (\mho^h \phi)^{-1}(g)) = y(\phi^{-1}(h^{-1}g)).$$
Removing the $y$'s, it suffices to show
$$\phi^{-1}(h^{-1}) (\mho^h \phi)^{-1}(g) = \phi^{-1}(h^{-1}g).$$
Multiply both sides by $\phi^{-1}(h^{-1})^{-1}$, to obtain the equivalent 
$$(\mho^h \phi)^{-1}(g) = \phi^{-1}(h^{-1})^{-1}\phi^{-1}(h^{-1}g).$$
Now apply $\mho^h \phi$ to both sides to obtain
$$g = (\mho^h \phi)(  \phi^{-1}(h^{-1})^{-1}\phi^{-1}(h^{-1}g)).$$
This is a straightforward consequence of the definition of $\mho^h\phi$ and so finishes our verification of (\ref{E:tcE}).

\begin{lem}\label{L:tau2}
Keep notation as in Lemma \ref{L:tau}. Also let $y \in \B^n$. Then 
$$\tcF(\tcE^{-1}(x^\s_v, y^\s_v)) = (x^\tau_v, y^\tau_v).$$
\end{lem}

\begin{proof}
Recall the notation $\phi_v = \cE^{-1}(x^\s_v) \in \sym_e(G)$. By definition of $\tcE$,
$$\tcE^{-1}(x^\s_v, y^\s_v) = (\phi_v, y^\s_v).$$
By definition of $\tcF$ and Lemma \ref{L:tau},
$$\tcF(\tcE^{-1}(x^\s_v, y^\s_v)) = (x^\tau_v, y^\s_v \circ \phi_v^{-1}).$$
So it suffices to show $y^\tau_v=y^\s_v \circ \phi_v^{-1}$. For any $g \in G$,
$$y^\tau_v(g) = y(\tau(g^{-1})v), \quad y^\s_v \circ \phi_v^{-1}(g) = y(\s(\phi_v^{-1}(g)^{-1})v).$$
The definition of $\tau$ gives $\tau(g^{-1})v = \s(\phi_v^{-1}(g)^{-1})v$. This finishes the proof.
\end{proof}

 \subsection{Proof of Theorem \ref{thm:main}}
 
 
 
 \begin{proof}[Proof of Theorem \ref{thm:main}]
 Without loss of generality, we assume the standard hypotheses  from \S \ref{S:prelim}. Because the actions $T$ and $U$ are boundedly orbit-equivalent, there is a $\rho \in \NN$ such that for $\mu$-a.e. $x\in X$ and every $s \in S \cup S^{-1}$, $\a(s,x)$ and $\beta(s,x)$ both have word-length $\le \rho$.

Let $\nu = \tGamma_*\mu \in \Prob(\sym_e(G) \times \B^G)$. Recall from \S \ref{sec:Gamma} that it suffices to show  $f_{\nu}(\tTheta) = f_{\nu}( \tilde{\mho})$. Recall the maps $\tcE$, $\tcF$ from \S \ref{sec:finitemodels}. These maps are injective (mod null sets) because $\cE$ and $\cF$ are injective (mod null sets). They are also equivariant in the sense of (\ref{E:tcE}). So it suffices to show $f_{\tcE_*\nu}(\A_\rho^G \times \B^G) = f_{\tcF_*\nu}(\A_\rho^G \times \B^G)$.

 Let $Z = \cE(\sym_\rho(G)) \subset \A_\rho^G$ and let $\tZ = Z \times \B^G$. Recall from \S \ref{S:symbolic} that $\Omega(\s,\cO)$ is the set of all $x\in \A_\rho^n \times \B^n$ such that $P^\s_x \in \cO$ where $P^\s_x$ is the empirical measure of $x$ (with respect to $\s$). Also $\Omega_\tZ(\s,\cO)$ is the set of $x \in \Omega(\s,\cO)$ such that $P^\s_x(\tZ)=1$. 
 
 Given an open subset $\cO \subset \Prob(\A_\rho^G\times \B^G)$, we define the following subsets of $\Hom(G,\sym(n))\times \A_\rho^n \times \B^n$:
 \begin{eqnarray*}
 \Omega(n,\cO) & := &  \{ (\s,x, y):~ (x,y) \in \Omega(\s,\cO)\} \\
   \Omega_\tZ(n,\cO) & := &  \{ (\s,x,y):~ (x,y) \in \Omega_\tZ(\s,\cO)\} \\
 \Omega_\tZ(n) & := &  \Omega_\tZ(n,\Prob(\A_\rho^G \times \B^G)).
 \end{eqnarray*} 
  By Theorem \ref{T:SFT1} $Z$ is a subshift of finite type. Therefore, $\tZ$ is also a subshift of finite type. By Theorem \ref{thm:SFT} and the formula $n!^r = \#\Hom(G,\sym(n))$,
\begin{eqnarray}
 f_{\tcE_*\nu}(\A_\rho^G \times \B^G) &=& \inf_{\cO \ni \tcE_*\nu} \limsup_{n \to \infty} \frac{1}{n} \log \EE_{\sigma \sim u_n} \abs{\Omega_\tZ(\sigma, \cO)}\nonumber \\
 &=& \inf_{\cO \ni \tcE_*\nu} \limsup_{n \to \infty} \frac{1}{n} \log\left(\frac{\#\Omega_\tZ(n,\cO)}{n!^r} \right).\label{E:2}
 \end{eqnarray}
 Define
$$\Upsilon_n:  \Omega_\tZ(n) \to \Hom(G,\sym(n))\times \A_\rho^n\times \B^n$$
by $\Upsilon_n(\s, x,y) = (\tau, x,y)$ where $\tau$ is as defined in Lemma \ref{L:tau}. To prove the theorem, we will show that $\Upsilon_n$ maps periodic points witnessing the $f$-invariant for $\tcE_*\nu$ to periodic points witnessing the $f$-invariant for $\tcF_*\nu$. The first step is the following continuity property.

\noindent {\bf Claim 1}. Let $\cU\subset \Prob(\A_\rho^G\times \B^G)$ be an open neighborhood of $\tcF_*\nu$. Then there exists an open neighborhood $\cO$ of $\tcE_*\nu$ such that $\Upsilon_n(\Omega_\tZ(n,\cO)) \subset \Omega(n,\cU).$

\begin{proof}[Proof of Claim 1]
The map  $\cE^{-1}:Z \to \sym_e(G)$ is continuous by Lemma \ref{L:inverse}. The continuity of $\cF, \tcF$ and $\tcE^{-1}:\tZ \to \sym_e(G) \times \B^G$ follows directly from their definitions. Therefore $\tcF\circ \tcE^{-1}:\tZ \to \A_\rho^G\times \B^G$ is continuous and the induced map
$$(\tcF\circ \tcE^{-1})_*:\Prob(\tZ) \to \Prob(\A_\rho^G\times \B^G)$$
is continuous. Since $\cU \subset \Prob(\A_\rho^G\times \B^G)$ is an open neighborhood of $\tcF_*\nu$, this implies that its pre-image in $\Prob(\tZ)$ is an open neighborhood of $\tcE_*\nu$. In particular, there exists an open neighborhood $\cO' \subset \Prob(\tZ)$ of $\tcE_*\nu$ such that $(\tcF\circ \tcE^{-1})_*(\cO') \subset \cU$. 

Because $\tZ$ is closed in $\A_\rho^G\times \B^G$, the weak* topology on $\Prob(\tZ)$ is the restriction of the weak* topology on $\Prob(\A_\rho^G\times \B^G)$. So there is an open neighborhood $\cO \subset \Prob(\A_\rho^G\times \B^G)$ of $\tcE_*\nu$ such that $\cO \cap \Prob(\tZ) \subset \cO'$. 

Let $(\s,x,y) \in \Omega_\tZ(n,\cO)$. Since $\Upsilon_n(\s,x, y)=(\tau,x, y)$, it suffices to show $(\tau,x,y) \in \Omega(n,\cU)$. Because $(\s,x,y) \in \Omega_\tZ(n,\cO)$,  $P^\s_{(x,y)}(\tZ)=1$ and $P^\s_{(x,y)} \in \cO$. So $P^\s_{(x,y)}\in \cO'$. Thus $(\tcF\circ \tcE^{-1})_*P^\s_{(x,y)} \in \cU$. However,
\begin{eqnarray*}
 (\tcF\circ \tcE^{-1})_* P^\s_{(x,y)} &=& \frac{1}{n} \sum_{v \in [n]}  (\tcF\circ \tcE^{-1})_* \d_{(x^\s_v, y^\s_v) } \\
  &=& \frac{1}{n} \sum_{v \in [n]}  \d_{(x^\tau_v, y^\tau_v)} = P^\tau_{(x,y)}.
 \end{eqnarray*}
Above we have used Lemma \ref{L:tau2} to conclude $(\cF\circ \cE^{-1})(x^\s_v, y^\s_v)=(x^\tau_v, y^\tau_v)$. So $P^\tau_{(x,y)} \in \cU$ which implies $(\tau,x,y) \in \Omega(n,\cU)$.

\end{proof}

\noindent {\bf Claim 2}.  $\Upsilon_n$ is injective.

\begin{proof}[Proof of Claim 2]
Suppose $\Upsilon_n(\s,x,y)=(\tau,x,y)$. Recall that $\phi_v = \cE^{-1}(x^\s_v) \in \sym_e(G)$. By Lemma \ref{L:tau},
$$\tau(g)v = \s(\phi^{-1}_v(g^{-1})^{-1})v.$$
Set $h = \phi^{-1}_v(g^{-1})^{-1}$. Then $g = \phi_v(h^{-1})^{-1}$. Thus
$$\s(h)v = \tau( \phi_v(h^{-1})^{-1} )v.$$
Now suppose that $h \in S \cup S^{-1}$. By Lemma \ref{L:inverse} with $g=e$,
$$\phi_v(h^{-1}) = (x^\s_v)_e(h^{-1}) = x(v)(h^{-1}).$$
Thus
$$\s(h)v = \tau(x(v)(h^{-1})^{-1})v$$
for all $h \in S\cup S^{-1}$ and $v\in [n]$. This shows that $\s$ is determined by $\tau$ and $x$. Thus $\Upsilon_n$ is injective.
\end{proof}

If $\cO, \cU$ are as in Claim 1, then Claim 2 implies
$$\# \Omega_\tZ(n,\cO) = \# \Upsilon_n(\Omega_\tZ(n,\cO)) \le \# \Omega(n,\cU).$$
So (\ref{E:2}) implies
\begin{eqnarray*}
 f_{\tcE_*\nu}(\A_\rho^G\times \B^G) &=&  \inf_{\cO \ni \tcE_*\nu} \limsup_{n \to \infty} \frac{1}{n} \log\left(\frac{\#\Omega_\tZ(n,\cO)}{n!^r} \right) \\
 &\le &  \inf_{\cU \ni \tcF_*\nu} \limsup_{n \to \infty} \frac{1}{n} \log\left(\frac{\#\Omega(n,\cU)}{n!^r} \right) \\
 &=& f_{\tcF_*\nu}(\A_\rho^G\times \B^G).
 \end{eqnarray*}
As mentioned above, since $\tcE$ and $\tcF$ are embeddings, this shows
$$f_\nu(\tTheta) = f_{\tcE_*\nu}(\A_\rho^G\times \B^G) \le  f_{\tcF_*\nu}(\A_\rho^G\times \B^G) = f_\nu(\tilde{\mho}).$$
By symmetry, the same argument with $T$ and $U$ switched shows the opposite inequality. This proves the theorem.

 \end{proof}

\appendix

\section{Open problems}\label{S:open}

\begin{enumerate}
\item Is there an analog of the Rudolph-Weiss Theorem concerning invariance of relative entropy under OE \cite{rudolph-weiss-2000} for sofic entropy or the $f$-invariant? The sofic-entropy formulation of the relative $f$-invariant in \cite{shriver-relativef} might be useful for this problem. If the Rudolph-Weiss Theorem generalizes to free groups then it should be possible to extend the $f$-invariant to actions of treeable groups via Hjorth's Lemma \cite{hjorth-cost-attained}.

\item Can the results of this paper be extended to free products of amenable groups or surface groups? The difficulty is that there is no analog of Theorem \ref{thm:SFT} in these cases. 

\item Suppose $G$ is a finitely generated infinite group. If two Bernoulli shifts over $G$ are bounded orbit equivalent then do they necessarily have the same base entropy? 

\item Is there some $1\le p <\infty$ such that the $f$-invariant is invariant under $\rL^p$-OE? 

\end{enumerate}


\bibliography{biblio}
\bibliographystyle{alpha}

\end{document}